\documentclass[reqno,12pt]{amsart}
\usepackage[margin=1.1in, footskip=1cm]{geometry}
\usepackage{amsmath, amsthm, amssymb, booktabs}
\usepackage{mathrsfs}
\usepackage[usenames,dvipsnames]{color}
\pdfpagewidth=\paperwidth
\pdfpageheight=\paperheight

\newtheorem{theorem}{Theorem}[section]
\newtheorem{lemma}[theorem]{Lemma}

\theoremstyle{definition}

\theoremstyle{remark}

\numberwithin{equation}{section}

\def\dbN{{\mathbb N}}

\def\del{{\delta}}

\def\lam{{\lambda}}

\def\le{\leqslant} \def\ge{\geqslant}

\makeatletter
\@namedef{subjclassname@2020}{\textup{2020} Mathematics Subject Classification}
\makeatother

\begin{document}
\title[Strong paucity]{Strong paucity in the Br\"udern-Robert\\ Diophantine system}
\author[T. D. Wooley]{Trevor D. Wooley}
\address{Department of Mathematics, Purdue University, 150 N. University Street, West 
Lafayette, IN 47907-2067, USA}
\email{twooley@purdue.edu}
\date{}
\dedicatory{In honorem Roger Heath-Brown LXXV nati}

\begin{abstract} Let $k$ be a natural number with $k\ge 2$, and let $\varepsilon>0$. We consider the number 
$V_k^*(P)$ of integral solutions of the system of simultaneous Diophantine equations
\[
x_1^{2j-1}+\ldots +x_{k+1}^{2j-1}=y_1^{2j-1}+\ldots +y_{k+1}^{2j-1}\quad (1\le j\le k),
\]
with $1\le x_i,y_i\le P$ $(1\le i\le k+1)$. Writing $L_k^*(P)$ for the number of diagonal solutions with 
$\{x_1,\ldots ,x_{k+1}\}=\{y_1,\ldots ,y_{k+1}\}$, so that $L_k^*(P)\sim (k+1)!P^{k+1}$, we prove that
\[
V_k^*(P)-L_k^*(P)\ll P^{\sqrt{8k+9}-1+\varepsilon}.
\]
This establishes a strong paucity result improving on earlier work of Br\"udern and Robert.
\end{abstract}

\maketitle

\section{Introduction} This memoir is devoted to the system of simultaneous Diophantine equations
\begin{equation}\label{1.1}
\sum_{i=1}^{2k+2}z_i^{2j-1}=0\quad (1\le j\le k).
\end{equation}
When $k$ is a natural number and $P$ is a large real number, we denote by $V_k(P)$ the number of integral 
solutions of the system \eqref{1.1} with $|z_i|\le P$ $(1\le i\le 2k+2)$. The situation with $k=1$ being too mundane 
to attract our attention, we focus instead on those scenarios in which $k\ge 2$. The system \eqref{1.1} possesses 
linear spaces of solutions having affine dimension $k+1$ typified by the one defined by the equations
\[
z_{2i-1}+z_{2i}=0\quad (1\le i\le k+1).
\]
On considering appropriate permutations of the underlying indices $\{1,2,\ldots ,2k+2\}$, a moment of reflection 
reveals that the number of such linear spaces is equal to
\[
c_k=2^{-k-1}\frac{(2k+2)!}{(k+1)!}=\prod_{j=1}^{k+1}(2j-1).
\]
We denote by $L_k(P)$ the number of integral solutions of the system \eqref{1.1} lying on the collection of all such 
linear spaces with $|z_i|\le P$ $(1\le i\le 2k+2)$. Thus, we have
\[
L_k(P)=c_k(2P+1)^{k+1}+O(P^k).
\]
We think of these solutions lying on such linear spaces as being {\it trivial}, and refer to any integral solution not 
of this type as being {\it non-trivial}. Our goal in this paper is to show that, when $k\ge 2$, there is a paucity of 
non-trivial solutions to the system \eqref{1.1} in a particularly strong sense.\par

In order to describe our main conclusion, we define the exponent
\begin{equation}\label{1.2}
\alpha_k=\min_{\substack{r\in \mathbb N\\ 2\le r\le 2k+1}}\Bigl( r-1+\frac{2k+2}{r}\Bigr) .
\end{equation} 

\begin{theorem}\label{theorem1.1}
When $k\ge 2$, one has
\[
V_k(P)-L_k(P)\ll P^{\alpha_k+\varepsilon}.
\]
In particular, one has
\[
V_k(P)-L_k(P)\ll P^{\sqrt{8k+9}-1+\varepsilon}.
\]
\end{theorem}

The conclusion of Theorem \ref{theorem1.1} shows that as $P\rightarrow \infty$, one has the asymptotic formula
\[
V_k(P)\sim c_k(2P+1)^k,
\]
whenever $k\ge 6$. This conclusion was obtained more than a decade ago by Br\"udern and Robert \cite{BR2012} 
(see the antepenultimate paragraph of the introduction of \cite{BR2012}), although the main focus of the latter 
authors was an analogue of Theorem \ref{theorem1.1} addressing the system \eqref{1.3} below in which the 
underlying variables are constrained to be positive. They show, in fact, that
\[
V_k(P)-L_k(P)\ll P^{\lambda_k+\varepsilon},
\]
where $\lambda_k<k+1$ for $k\ge 2$, and in particular $\lambda_k=\frac{6}{7}k+O(1)$ as $k\rightarrow \infty$. For 
comparison, our new estimate recorded in Theorem \ref{theorem1.1} shows that one may take 
$\lambda_k=\sqrt{8k+9}-1$, which for large values of $k$ is considerably sharper than the bound of Br\"udern and 
Robert. We note in this context that our own work would address small values of $k$ successfully, though not in a
way superior to the work of Br\"udern and Robert without substantially greater effort. In particular, their admissible 
exponent $\lambda_3=34/9$ would appear to be sharper than conclusions immediately available through variants of 
the ideas introduced in this paper. Finally, we note that the work of Vaughan and Wooley \cite{VW1995} shows 
that
\[
V_2(P)-L_2(P)\asymp P^2(\log P)^5,
\]
a relation that has been refined to give an asymptotic formula by la Bret\`eche \cite{dlB2007}.\par

For the sake of completeness, we record a version of Theorem \ref{theorem1.1} relevant for the analogous problem 
in which the variables are restricted to take positive values. We denote by $V_k^*(P)$ the number of integral 
solutions of the system
\begin{equation}\label{1.3}
\sum_{i=1}^{k+1}x_i^{2j-1}=\sum_{i=1}^{k+1}y_i^{2j-1}\quad (1\le j\le k),
\end{equation}
with $1\le x_i,y_i\le P$ $(1\le i\le k+1)$. Also, we denote by $L_k^*(P)$ the number of integral $(k+1)$-tuples 
$\mathbf x,\mathbf y$ in which $x_1,\ldots ,x_{k+1}$ is a permutation of $y_1,\ldots ,y_{k+1}$ and $1\le x_i,y_i\le P$ 
$(1\le i\le k+1)$. Thus, we have
\[
L_k^*(P)=(k+1)!P^{k+1}+O(P^k).
\]

\begin{theorem}\label{theorem1.2}
When $k\ge 2$, one has
\[
V_k^*(P)-L_k^*(P)\ll P^{\alpha_k+\varepsilon}.
\]
In particular, one has
\[
V_k^*(P)-L_k^*(P)\ll P^{\sqrt{8k+9}-1+\varepsilon}.
\]
\end{theorem}

Our approach to the proof of Theorem \ref{theorem1.1} is based on that employed in our work joint with Vaughan 
\cite{VW1997} concerning the corresponding Vinogradov system
\begin{equation}\label{1.4}
\sum_{i=1}^{k+1}x_i^j=\sum_{i=1}^{k+1}y_i^j\quad (1\le j\le k).
\end{equation}
Write $J_{k+1,k}(P)$ for the number of solutions of the system \eqref{1.4} with $1\le x_i,y_i\le P$ $(1\le i\le k+1)$. 
Then Vaughan and Wooley \cite[Theorem 1]{VW1997} shows that
\[
J_{k+1,k}(P)-L_k^*(P)\ll P^{\sqrt{4k+5}+\varepsilon}.
\]
This bound is achieved by showing that the solutions of the system \eqref{1.4} satisfy auxiliary equations exhibiting 
copious multiplicative structure. The latter fosters a parameterisation of solutions that permits considerable control 
to be exercised over the number of non-trivial solutions. A similar approach is possible for the system \eqref{1.1}, as 
will be apparent from the discussion of \S2. The exploitation of this parameterisation of the solutions will be 
discussed in \S3, where we complete the proofs of Theorems \ref{theorem1.1} and \ref{theorem1.2}. We report on 
various ideas for refinement of our main theorem in \S4. In \S5 we describe a strong paucity result for a generalisation 
of the system \eqref{1.1} that further illustrates the strategy employed in the proof of Theorem \ref{theorem1.1}. 
Finally, in \S6, we record for future use a certain elementary discrete inequality applied in proving the second 
assertion of Theorem \ref{theorem1.1}.\par

Our basic parameter is $P$, a sufficiently large positive integer. Whenever $\varepsilon$ appears in a statement, 
either implicitly or explicitly, we assert that the statement holds for each $\varepsilon>0$. Throughout, the symbols 
$\ll $ and $\gg $ denote Vinogradov's well-known notation. Implicit constants in both the notations of Vinogradov 
and Landau will depend at most on $\varepsilon$, $k$ and $r$. We make frequent use of vector notation in the form 
$\mathbf x=(x_1,\ldots ,x_r)$. Here, the dimension $r$ depends on the course of the argument.\par

Work on this paper was conducted while the author was supported by NSF grant DMS-2502625 and Simons 
Fellowship in Mathematics SFM-00011955. The author is grateful to the Institute for Advanced Study, Princeton, for 
hosting his sabbatical, during which period this paper was completed.\par

\noindent{\bf Historical note:} The author's interest in and work on the topic of this memoir can be traced back to his 
Ph.D.~studies, at which time the author came across the paper of Roger Heath-Brown \cite{HB1988} making use of 
quasi-paucity estimates for the Diophantine system
\[
\left .\begin{aligned}x_1^3+x_2^3+x_3^3&=y_1^3+y_2^3+y_3^3\\
x_1+x_2+x_3&=y_1+y_2+y_3\end{aligned}\right\}.
\]
This inspired work in the author's thesis \cite[Theorem 1.3 of Chapter 3]{Woo1990}, completed under the supervision 
of Bob Vaughan in 1990, providing analogous conclusions concerning pairs of equations having arbitrary degrees. The 
author is grateful to Roger for this, and many other contributions, that have provided such inspiration for nearly four 
decades. 

\section{Multiplicative identities}
We make use of the polynomial identities prepared in work of Br\"udern and Robert \cite{BR2012}, though in a more 
symmetrical guise better suited to our subsequent arguments. Define the polynomials $t_j(\mathbf x)$ by putting
\[
t_j(\mathbf x)=\sum_{i=1}^{k-1}x_i^{2j-1}\quad (1\le j\le k).
\]
Then we find from Br\"udern and Robert \cite[Lemma 1]{BR2012} that there exists a non-zero polynomial 
$\Upsilon (\mathbf w)\in \mathbb Z[w_1,\ldots ,w_k]$ having the property that
\[
\Upsilon (t_1(\mathbf x),\ldots ,t_k(\mathbf x))=0.
\]
The weighted degree of $\Upsilon (\mathbf w)$ is $k(k+1)/2$, in the sense that in each monomial of 
$\Upsilon (\mathbf w)$ of the shape $c_{\boldsymbol \alpha}w_1^{\alpha_1}\cdots w_k^{\alpha_k}$, with 
$c_{\boldsymbol \alpha}\in \mathbb Z$ and $\alpha_i\ge 0$ $(1\le i\le k)$, one has
\begin{equation}\label{2.1}
\sum_{i=1}^k(2i-1)\alpha_i=k(k+1)/2.
\end{equation}

\par Next, define the polynomials $\tau_j(\mathbf x)$ by putting
\begin{equation}\label{2.2}
\tau_j(\mathbf x)=\sum_{i=1}^{k+1}x_i^{2j-1}\quad (1\le j\le k).
\end{equation}
Then, as in the corresponding discussion of \cite{BR2012}, we find that when we make the specialisation 
$x_k+x_{k+1}=0$, one has
\[
\Upsilon(\tau_1(\mathbf x),\ldots ,\tau_k(\mathbf x))=\Upsilon (t_1(\mathbf x),\ldots ,t_k(\mathbf x))=0,
\]
so that the polynomial $\Upsilon (\tau_1(\mathbf x),\ldots ,\tau_k(\mathbf x))$ is divisible by $x_k+x_{k+1}$. On 
noting \eqref{2.1}, symmetrical arguments reveal that there exists a non-zero integer $C=C_k$ having the property 
that
\begin{equation}\label{2.3}
\Upsilon(\tau_1(\mathbf x),\ldots ,\tau_k(\mathbf x))=C\prod_{1\le i<j\le k+1}(x_i+x_j).
\end{equation}

\par Suppose that $\mathbf z\in \mathbb Z^{2k+2}$ is a solution of the system of equations \eqref{1.1}. We write 
$z_i'=z_{k+1+i}$ $(1\le i\le k+1)$. Then, on utilising the notation \eqref{2.2}, we have
\[
\tau_j(z_1\ldots ,z_{k+1})=\tau_j(-z_1',\ldots ,-z_{k+1}')\quad (1\le j\le k).
\]
Hence, we see from \eqref{2.3} that
\begin{align*}
C\prod_{1\le i<j\le k+1}(z_i+z_j)&=\Upsilon(\tau_1(\mathbf z),\ldots ,\tau_k(\mathbf z))\\
&=\Upsilon(\tau_1(-\mathbf z'),\ldots ,\tau_k(-\mathbf z'))\\
&=(-1)^{k(k+1)/2}C\prod_{1\le i<j\le k+1}(z_i'+z_j').
\end{align*}
We therefore deduce that
\begin{equation}\label{2.4}
\prod_{1<l\le k+1}(z_1+z_l)\prod_{2\le i<j\le k+1}(z_i+z_j)
=(-1)^{k(k+1)/2}\prod_{1<l\le k+1}(z_1'+z_l')\prod_{2\le i<j\le k+1}(z_i'+z_j').
\end{equation}
On interchanging the roles of $z_1$ and $z_1'$, we find in like manner that
\begin{equation}\label{2.5}
\prod_{1<l\le k+1}(z_1'+z_l)\prod_{2\le i<j\le k+1}(z_i+z_j)
=(-1)^{k(k+1)/2}\prod_{1<l\le k+1}(z_1+z_l')\prod_{2\le i<j\le k+1}(z_i'+z_j').
\end{equation}
The multiplicative relations \eqref{2.4} and \eqref{2.5} permit the extraction of much simpler relations that occur, in 
slightly different guise, also in the work of Br\"udern and Robert \cite[Lemma 3]{BR2012}. We provide an account of 
the derivation of these relations for the sake of completeness.

\begin{lemma}\label{lemma2.1}
Suppose that $\mathbf z\in \mathbb Z^{2k+2}$ satisfies the system of equations \eqref{1.1}, and one has 
$z_i+z_j\ne 0$ $(1\le i<j\le 2k+2)$. Then one has
\begin{equation}\label{2.6}
\prod_{i=3}^{2k+2}(z_1+z_i)=\prod_{i=3}^{2k+2}(z_2+z_i).
\end{equation}
\end{lemma}

\begin{proof}
In view of the non-vanishing condition $z_i+z_j\ne 0$ imposed in the hypotheses of the lemma, we find by 
dividing left and right hand sides of \eqref{2.4} and \eqref{2.5} that
\[
\prod_{1<l\le k+1}\Bigl( \frac{z_1+z_l}{z_1'+z_l}\Bigr) =\prod_{1<l\le k+1}\Bigl( \frac{z_1'+z_l'}{z_1+z_l'}\Bigr) ,
\]
whence
\[
\prod_{1<l\le k+1}(z_1+z_l)(z_1+z_l')=\prod_{1<l\le k+1}(z_1'+z_l')(z_1'+z_l).
\]
Recall that $z_l'=z_{k+1+l}$ $(1\le l\le k+1)$. Then we see that
\[
\prod_{\substack{1\le l\le 2k+2\\ l\not\in \{1,k+2\}}}(z_1+z_l)=
\prod_{\substack{1\le l\le 2k+2\\ l\not\in \{1,k+2\}}}(z_{k+2}+z_l).
\]
The conclusion of the lemma follows on applying symmetry to interchange the roles of $z_{k+2}$ and $z_2$.
\end{proof}

We seek to create a multitude of interacting multiplicative relations with which to constrain the integer tuple 
$\mathbf z$. In order to motivate the system of equations with which we shall work, we begin by observing that the 
equation \eqref{2.6} may be written in the more symmetrical form
\[
\biggl( \prod_{\substack{1\le i\le r\\ i\ne 1}}z_i\biggr) \biggl( \prod_{j=1}^{2k+2}(z_1+z_j)\biggr) =
\biggl( \prod_{\substack{1\le i\le r\\ i\ne 2}}z_i\biggr) \biggl( \prod_{j=1}^{2k+2}(z_2+z_j)\biggr) .
\]
We now introduce the parameter $r\in \mathbb N$ satisfying $2\le r\le 2k+2$, and apply symmetry to obtain the 
system of simultaneous equations
\begin{equation}\label{2.7}
\biggl( \prod_{\substack{1\le i\le r\\ i\ne l}}z_i\biggr) \biggl( \prod_{j=1}^{2k+2}(z_l+z_j)\biggr) =
\biggl( \prod_{\substack{1\le i\le r\\ i\ne m}}z_i\biggr) \biggl( \prod_{j=1}^{2k+2}(z_m+z_j)\biggr) \quad 
(1\le l<m\le r).
\end{equation}

\par In order better to disentangle the multiplicative structure of the system of equations \eqref{2.7}, when 
$1\le m\le r$, we introduce the integers
\begin{equation}\label{2.8}
u_{0m}=\biggl( \prod_{\substack{1\le i\le r\\ i\ne m }}z_i\biggr) \biggl( \prod_{j=1}^r(z_m+z_j)\biggr) 
\end{equation}
and
\begin{equation}\label{2.9}
u_{lm}=z_m+z_{r+l}\quad (1\le l\le 2k+2-r).
\end{equation}
For the sake of concision, we write
\begin{equation}\label{2.10}
\kappa=2k+2-r.
\end{equation}
The system of equations \eqref{2.7} now assumes the shape
\begin{equation}\label{2.11}
\prod_{i_1=0}^\kappa u_{i_11}=\prod_{i_2=0}^\kappa u_{i_22}=\ldots =\prod_{i_r=0}^\kappa u_{i_rr}.
\end{equation}
Note that the equations \eqref{2.9} imply that
\[
u_{lm}-z_m=z_{r+l}\quad (1\le l\le \kappa),
\]
whence
\begin{equation}\label{2.12}
u_{i1}-z_1=u_{i2}-z_2=\ldots =u_{ir}-z_r\quad (1\le i\le \kappa).
\end{equation}

\par Given a fixed choice of the integer $r$-tuple $\mathbf z=(z_1,\ldots ,z_r)$, we denote by 
$\Psi_k(P;\mathbf z)$ the number of integral solutions of the simultaneous equations \eqref{2.11} and \eqref{2.12} 
with
\[
1\le |u_{ij}|\le 2P\quad (1\le i\le \kappa,\, 1\le j\le r).
\] 
We then define
\[
\Psi_k(P)=\max_{\mathbf z}\Psi_k(P;\mathbf z),
\]
where the maximum is taken over all $r$-tuples $\mathbf z$ satisfying
\begin{equation}\label{2.13}
1\le |z_i|\le P\quad (1\le i\le r)\qquad \text{and}\qquad z_l^2\ne z_m^2\quad (1\le l<m\le r).
\end{equation}

\begin{lemma}\label{lemma2.2}
Let $r$ be an integer with $2\le r\le 2k+1$. Then
\[
V_k(P)-L_k(P)\ll P^{r-1}+P^r\Psi_k(P).
\]
\end{lemma}

\begin{proof} We divide the solutions $\mathbf z=(z_1,\ldots ,z_{2k+2})$ of the system of equations \eqref{1.1} 
counted by $V_k(P)-L_k(P)$ into three types.\par

First, we consider solutions $\mathbf z$ of \eqref{1.1} counted by $V_k(P)-L_k(P)$ in which $z_l+z_m=0$ for some 
indices $l$ and $m$ with $1\le l<m\le 2k+2$. By relabelling variables, if necessary, we may suppose that there is an 
integer $u$ with $0\le u\le k-1$ having the property that
\[
z_{2i-1}+z_{2i}=0\quad (u+2\le i\le k+1),
\]
and so that
\begin{equation}\label{2.14}
z_l+z_m\ne 0\quad (1\le l<m\le 2u+2).
\end{equation}
One then has
\begin{equation}\label{2.15}
\sum_{i=1}^{2u+2}z_i^{2j-1}=0\quad (1\le j\le k).
\end{equation}
We consider two possible scenarios. In the first scenario, we have $2u+2\ge k+1$. Here, we write
\[
\mathbf v=(z_1,\ldots ,z_{k+1})\quad \text{and}\quad \mathbf w=(z_{k+2},\ldots ,z_{2u+2},0,\ldots ,0).
\]
Since $u\le k-1$, we find from the identity \eqref{2.3} that
\[
\Upsilon(\tau_1(\mathbf v),\ldots ,\tau_k(\mathbf v))=
(-1)^{k(k+1)/2}\Upsilon(\tau_1(\mathbf w),\ldots ,\tau_k(\mathbf w))=0,
\]
whence
\[
\prod_{1\le l<m\le k+1}(z_l+z_m)=0.
\]
However, the condition \eqref{2.14} prohibits such an eventuality. When instead $2u+2<k+1$, we may proceed 
similarly on replacing the system \eqref{2.15} by the implied subsystem
\[
\sum_{i=1}^{2u+2}z_i^{2j-1}=0\quad (1\le j\le 2u+1).
\]
In this scenario, an analogue of \eqref{2.3} implies that
\[
\prod_{1\le l<m\le 2u+2}(z_l+z_m)=0,
\]
and we arrive at a similar contradiction. We therefore conclude that $z_l+z_m\ne 0$ for $1\le l<m\le 2k+2$ in all 
solutions $\mathbf z$ of \eqref{1.1} counted by $V_k(P)-L_k(P)$. Notice that this conclusion also ensures that there 
can exist at most one index $i$ with $1\le i\le 2k+2$ having the property that $z_i=0$.\par

The second type of solution $\mathbf z$ of the system \eqref{1.1} counted by $V_k(P)-L_k(P)$ is that in which there 
are fewer than $r$ distinct values of $z_i$ with $1\le i\le 2k+2$. It is evident that the total number of such solutions is 
$O(P^ {r-1})$.\par

We are left with the third scenario, which concerns the solutions $\mathbf z$ of the system \eqref{1.1} counted by 
$V_k(P)-L_k(P)$ of neither the first nor the second type. In this situation, we may relabel variables in such a manner 
that $z_i\ne 0$ for $1\le i\le r$. Further, we can suppose that $z_1,\ldots ,z_r$ are pairwise distinct, and that 
$z_l+z_m\ne 0$ for $1\le l<m\le 2k+2$. In particular, we may suppose that the conditions \eqref{2.13} are all in play. 
Each such solution $\mathbf z$ of the system \eqref{1.1} generates a solution $\mathbf u$ of the simultaneous 
equations \eqref{2.11} and \eqref{2.12}. We fix any one of the $O(P^r)$ possible choices for $z_1,\ldots ,z_r$ 
satisfying \eqref{2.13}. By virtue of the relations \eqref{2.8}, this choice fixes the integers $u_{0m}$ $(1\le m\le r)$ with
\[
0<|u_{0m}|\le 2^rP^{2r-1}.
\]
Consider next a fixed solution $u_{lm}$ $(1\le l\le \kappa,\, 1\le m\le r)$ of the system of equations \eqref{2.11} and 
\eqref{2.12} corresponding to this fixed choice of $z_1,\ldots ,z_r$. There are at most $O(\Psi_k(P))$ possible such 
choices. But for each fixed choice of
\[
z_1,\ldots ,z_r\qquad \text{and}\qquad u_{lm}\quad (1\le l\le \kappa,\, 1\le m\le r),
\]
it follows from \eqref{2.9} that the integers $z_{r+l}$ $(1\le l\le \kappa)$ are fixed. This fixes $z_i$ for $1\le i\le 2k+2$, 
and thus we deduce that the total number of solutions of this third type is $O(P^r\Psi_k(P))$.\par

By combining the contributions from the three types of solutions $\mathbf z$ of the system \eqref{1.1} counted by 
$V_k(P)-L_k(P)$, we obtain the conclusion of the lemma, completing our proof.
\end{proof}

\section{Multiplicative parameterisations}
We turn in this section to the problem of extracting and utilising multiplicative structure available from the relations 
\eqref{2.11}. Here, we follow very closely the path laid down in \cite[\S2]{VW1997} as modified and adapted in 
\cite[\S4]{Woo2023}. Our present situation being different in detail, we provide a relatively complete account of this 
approach.\par

\begin{proof}[Proof of Theorem \ref{theorem1.1}] We begin by fixing $z_1,\ldots ,z_r$ satisfying the conditions 
\eqref{2.13} in such a manner that
\[
\Psi_k(P;\mathbf z)=\Psi_k(P).
\]
Next, we introduce some notation with which to describe precisely the extraction of common factors between 
variables. Let $\mathcal I$ denote the set of indices $\mathbf i=(i_1,\ldots ,i_r)$ with $0\le i_m\le \kappa$ 
$(1\le m\le r)$. We define the map $\phi:\mathcal I\rightarrow [0,(\kappa+1)^r)\cap \mathbb Z$ by putting
 \[
\phi(\mathbf i)=\sum_{m=1}^ri_m(\kappa+1)^{m-1}.
\]
Then $\phi$ is bijective, and we may define the successor $\mathbf i+1$ of the index $\mathbf i$ by setting
\[
\mathbf i+1=\phi^{-1}\bigl( \phi(\mathbf i)+1\bigr) .
\]
When $h\in \mathbb N$, we define $\mathbf i+h$ inductively by putting $\mathbf i+(h+1)=(\mathbf i+h)+1$. 
Finally, we write $\mathcal J(\mathbf i)$ for the set of $r$-tuples $\mathbf j\in \mathcal I$ having the property that 
for some $h\in \mathbb N$, one has $\mathbf j+h=\mathbf i$.\par

We now define the integers $\alpha_{\mathbf i}$, with $\mathbf i\in \mathcal I$, by defining greatest common divisors 
amongst $r$-tuples of integers as follows. We put $\alpha_{\mathbf 0}=(u_{01},u_{02},\ldots ,u_{0r})$, and we 
suppose at stage $\mathbf i$ that $\alpha_{\mathbf j}$ has been defined already for 
$\mathbf j\in \mathcal J(\mathbf i)$. We then define $\beta_{\mathbf i}^{(m)}$ by putting
\[
\beta_{\mathbf i}^{(m)}=\prod_{\substack{\mathbf j\in \mathcal J(\mathbf i)\\ j_m=i_m}}\alpha_{\mathbf j}\quad 
(1\le m\le r),
\]
and then define
\[
\alpha_{\mathbf i}=\biggl( \frac{u_{i_11}}{\beta_{\mathbf i}^{(1)}}, \frac{u_{i_22}}{\beta_{\mathbf i}^{(2)}},\ldots, 
\frac{u_{i_rr}}{\beta_{\mathbf i}^{(r)}}\biggr).
\]
Here, as is usual in voidology, we construe the empty product as being equal to unity. Write
\begin{equation}\label{3.1}
{\widetilde \alpha}_{lm}^\pm =\pm \prod_{\substack{\mathbf i\in \mathcal I\\ i_m=l}}\alpha_{\mathbf i}.
\end{equation}
Then, it follows that when $0\le l\le \kappa$ and $1\le m\le r$, there is a choice of the sign $\pm$ for which one has 
$u_{lm}={\widetilde \alpha}_{lm}^\pm$. We note that when $1\le m\le r$, in view of our constraints on $u_{lm}$, one 
has
\[
1\le |{\widetilde \alpha}_{0m}^\pm|\le 2^rP^{2r-1}\qquad \text{and}\qquad 1\le |{\widetilde \alpha}_{lm}^\pm |\le 2P
\quad (1\le l\le \kappa).
\]

\par We consider the variables $\alpha_{\mathbf i}$, with $\mathbf i\in \mathcal I$, to be variables, noting that 
${\widetilde \alpha}_{0m}^\pm=u_{0m}$ is already fixed by our choices for $z_1,\ldots ,z_r$. Thus, we find via 
standard divisor function estimates that for each index $m$ with $1\le m\le r$, the number of possible choices for 
$\alpha_{\mathbf j}$ with $\mathbf j\in \mathcal I$ and $j_m=0$ is $O(P^\varepsilon)$. We fix any one such choice 
for these particular variables $\alpha_{\mathbf j}$ with $j_m=0$, as well as a choice of sign giving 
${\widetilde \alpha}_{0m}^\pm=u_{0m}$. At this point, we see that $\Psi_k(P;\mathbf z)\ll X_r(P)$, where $X_r(P)$ 
denotes the number of integral solutions of the system
\begin{equation}\label{3.2}
{\widetilde \alpha}_{i1}^\pm -z_1={\widetilde \alpha}_{i2}^\pm -z_2=\ldots ={\widetilde \alpha}_{ir}^\pm -z_r
\quad (1\le i\le \kappa),
\end{equation}
with
\begin{equation}\label{3.3}
1\le |{\widetilde \alpha}_{ij}^\pm|\le 2P\quad (1\le i\le \kappa,\, 1\le j\le r).
\end{equation}
Here, we recall that we may assume that the conditions \eqref{2.13} remain in play, so that $z_l^2\ne z_m^2$ 
$(1\le l<m\le r)$ and $z_i\ne 0$ $(1\le i\le r)$. We deduce from Lemma \ref{lemma2.2} that
\begin{equation}\label{3.4}
V_k(P)-L_k(P)\ll P^{r-1}+P^{r+\varepsilon}X_r(P).
\end{equation}

\par We may now proceed as in the analogous discussion of \cite[\S2]{VW1997}, taking inspiration from the 
modifications adopted in \cite[\S4]{Woo2023}. When $1\le p\le r$, we put
\[
A_p=\prod_{\substack{\mathbf i\in \mathcal I\\ i_l>i_p\ge 1\, (l\ne p)}}\alpha_{\mathbf i}.
\]
Then we have
\[
\biggl| \prod_{p=1}^rA_p\biggr|\le \prod_{\substack{\mathbf i\in \mathcal I\\ i_m\ge 1\, (1\le m\le r)}}
|\alpha_{\mathbf i}|\le \prod_{l=1}^\kappa |{\widetilde \alpha}_{l1}^\pm|\le (2P)^\kappa .
\] 
Consequently, in any solution ${\boldsymbol \alpha}$ counted by $X_r(P)$, there exists a choice for the index $p$ 
with $1\le p\le r$ for which one has
\[
1\le |A_p|\le (2P)^{\kappa/r}.
\]
Furthermore, given an index $m$ with $1\le m\le r$, it follows from \eqref{3.2} that, for each solution 
${\boldsymbol \alpha}$ counted by $X_r(P)$, when $1\le j\le r$ and $j\ne m$, one has
\[
{\widetilde \alpha}_{im}^\pm -{\widetilde \alpha}_{ij}^\pm=z_m-z_j\quad (1\le i\le \kappa).
\]
By relabelling variables, if necessary, it follows that $X_r(P)\ll Y_r(P)$, where $Y_r(P)$ denotes the number of solutions 
of the system of equations 
\begin{equation}\label{3.5}
{\widetilde \alpha}_{im}^\pm -{\widetilde \alpha}_{i1}^\pm=z_m-z_1\quad (2\le m\le r,\, 1\le i\le \kappa),
\end{equation}
with $z_m-z_1$ fixed and non-zero with $|z_m-z_1|\le 2P$, and with the $\alpha_{\mathbf i}$ satisfying \eqref{3.3} 
and 
the inequality
\begin{equation}\label{3.6}
1\le |A_1|\le (2P)^{\kappa/r}.
\end{equation}
Furthermore, by \eqref{3.4}, we have
\begin{equation}\label{3.7}
V_k(P)-L_k(P)\ll P^{r-1}+P^{r+\varepsilon}Y_r(P).
\end{equation}

\par We claim that when the variables $\alpha_{\mathbf i}$ with
\begin{equation}\label{3.8}
\mathbf i\in \mathcal I\qquad \text{and}\qquad i_m>i_1\quad (2\le m\le r)
\end{equation}
are fixed, then there are $O(P^\varepsilon)$ possible choices for the variables $\alpha_{\mathbf i}$ satisfying 
\eqref{3.3} and \eqref{3.5}. Supposing temporarily such to be the case, the combination of \eqref{3.6} and \eqref{3.7}, 
together with a standard estimate for the divisor function, shows that
\begin{equation}\label{3.9}
V_k(P)-L_k(P)\ll P^{r-1}+P^{r+\kappa/r+\varepsilon}.
\end{equation}
On recalling from \eqref{2.10} that $\kappa=2k+2-r$, the first conclusion of Theorem \ref{theorem1.1} follows by 
reference to the definition \eqref{1.2}.\par

We confirm this claim by induction as follows. For a fixed choice of the variables $\alpha_{\mathbf i}$ with indices 
$\mathbf i$ satisfying \eqref{3.8}, we suppose at step $t$ in this induction that there are $O(P^{t\varepsilon})$ 
possible choices for those variables $\alpha_{\mathbf i}$ for which the index $\mathbf i$ satisfies the condition that 
$i_m<t$ for some index $m$ with $1\le m\le r$. This conclusion is immediate when $t=1$, since the variables 
$\alpha_{\mathbf i}$ with $\mathbf i\in \mathcal I$ and $i_m=0$, for some index $m$ with $1\le m\le r$, are already 
fixed. Suppose then that the inductive hypothesis is satisfied for a value of $t$ with $t\ge 1$, and consider a fixed 
one of the $O(P^{t\varepsilon})$ possible choices for the variables $\alpha_{\mathbf i}$ for which $i_m<t$ for some 
index $m$ with $1\le m\le r$. It follows from \eqref{3.5} that
\begin{equation}\label{3.10}
{\widetilde \alpha}_{tm}^\pm ={\widetilde \alpha}_{t1}^\pm +z_m-z_1\quad (2\le m\le r).
\end{equation}
We investigate the product ${\widetilde \alpha}_{t1}^\pm$ defined via \eqref{3.1}, noting that if
\begin{equation}\label{3.11}
i_1=t\qquad \text{and}\qquad i_m\ne t\quad (2\le m\le r),
\end{equation}
then either $i_m<t$ for some index $m$ with $2\le m\le r$, or else $i_m>t$ for all indices $m$ with $2\le m\le r$. In 
the former situation, the variable $\alpha_{\mathbf i}$ is already fixed via the assumption preceding \eqref{3.10} that 
is a consequence of the inductive hypothesis. In the latter situation, meanwhile, the variables $\alpha_{\mathbf i}$ are 
fixed by virtue of our hypothesis \eqref{3.8}. Hence the variables $\alpha_{\mathbf i}$ for which $\mathbf i$ satisfies 
\eqref{3.11} may also be supposed to be fixed. Furthermore, by reference to \eqref{3.1}, we see that for each index 
$m$ with $2\le m\le r$, the variables $\alpha_{\mathbf i}$ for which $i_1=i_m=t$ all divide both 
${\widetilde \alpha}_{t1}^\pm$ and ${\widetilde \alpha}_{tm}^\pm$. Consequently, by \eqref{3.10}, these variables 
also divide the respective fixed non-zero integers $z_m-z_1$. We therefore infer by means of an elementary divisor 
function estimate that there are $O(P^\varepsilon)$ possible choices for the variables $\alpha_{\mathbf i}$ satisfying 
the condition that $i_1=i_m=t$ for some index $m$ with $2\le m\le r$. Fixing any one of these choices, we may 
suppose at this point that $\alpha_{\mathbf i}$ is fixed whenever $i_1=t$. In view of \eqref{3.1}, we may thus suppose 
that ${\widetilde \alpha}_{t1}^\pm$ is now fixed, and it follows from \eqref{3.10} that ${\widetilde \alpha}_{tm}^\pm$ 
is also fixed for $2\le m\le r$. Again employing \eqref{3.1}, in combination with an elementary estimate for the 
divisor function, we infer that $\alpha_{\mathbf i}$ is fixed whenever $i_m=t$ for some index $m$ with $1\le m\le r$. 
We may therefore conclude that there are $O(P^{(t+1)\varepsilon})$ possible choices for the variables 
$\alpha_{\mathbf i}$ for which $\mathbf i$ satisfies the condition that $i_m\le t$ for some index $m$ with 
$1\le m\le r$. This confirms the inductive hypothesis with $t$ replaced by $t+1$. This completes the induction, the 
case $t=\kappa+1$ of which establishes the claim that we employed in our proof of the upper bound \eqref{3.9}. 
This completes the proof of the first conclusion of Theorem \ref{theorem1.1}.\par

Equipped with the first conclusion of Theorem \ref{theorem1.1}, we have
\[
V_k(P)-L_k(P)\ll P^{r+(2k+2-r)/r+\varepsilon}.
\]
However, the elementary discrete inequality provided in Theorem \ref{theorem6.1} below shows that
\[
\min_{2\le r\le 2k+1}\Bigl( r+\frac{2k+2}{r}\Bigr) \le \sqrt{4(2k+2)+1}.
\]
The second conclusion of Theorem \ref{theorem1.1} is now immediate.
\end{proof}

The proof of Theorem \ref{theorem1.2} follows from that of Theorem \ref{theorem1.1} by a trivial specialisation.

\begin{proof}[Proof of Theorem \ref{theorem1.2}]
Since the quantity $V_k^*(P)-L_k^*(P)$ counts the non-trivial solutions of the system \eqref{1.1} counted by 
$V_k(P)-L_k(P)$ with $z_1,\ldots ,z_{k+1}$ positive and $z_{k+2},\ldots ,z_{2k+2}$ negative, we find by means of 
Theorem \ref{theorem1.1} that
\[
0\le V_k^*(P)-L_k^*(P)\le V_k(P)-L_k(P)\ll P^{\alpha_k+\varepsilon}\ll P^{\sqrt{8k+9}-1+\varepsilon}.
\]
The conclusions of Theorem \ref{theorem1.2} follow from these inequalities.
\end{proof}

\section{Speculative and realisable improvements}
We take the opportunity in this section to explore the potential for improvement in the paucity estimates supplied by 
Theorems \ref{theorem1.1} and \ref{theorem1.2}. Accessible improvements are of interest, of course, in their own 
right. Improvements achievable subject to plausible hypotheses, on the other hand, offer insights concerning 
plausible avenues for future investigation, and perhaps an indication of limitations on the scope of that progress.\par

We begin by exploring more carefully the potential for exploiting the equations \eqref{2.11} and \eqref{2.12} more 
efficiently. In the argument presented in \S3 that establishes Theorem \ref{theorem1.1}, we fixed $z_1,\ldots ,z_r$, 
subsequently absorbed these fixed integers into the factors
\[
u_{0m}=\biggl( \prod_{\substack{1\le i\le r\\ i\ne m}}z_i\biggr) \biggl( \prod_{j=1}^r(z_m+z_j)\biggr) \quad 
(1\le m\le r)
\]
occurring in the relation \eqref{2.11}, but then made no use of any constraints implicitly generated by these factors 
on the remaining variables. In a sense, these fixed variables are carried through the ensuing argument as dead 
weight. One might reasonably expect that the greatest common divisors $\alpha_{\mathbf i}$ occurring as factors of 
these integers $u_{0m}$ would contain their fair share of the weight relative to the mass of all the 
$\alpha_{\mathbf j}$ dividing the collection of all variables $u_{lm}$ with $l\ge 1$. After all, the extreme situation 
contrary to this supposition would be that in which, for $j\ge 1$, one has $(u_{0j},u_{lm})=1$ for $l\ge 1$ and 
$1\le m\le r$. In this situation, one must have
\[
\prod_{\substack{1\le l\le r\\ l\not \in \{i,j\}}}(z_i+z_l)=\pm \prod_{\substack{1\le l\le r\\ l\not \in \{i,j\}}}
(z_j+z_l)\quad (1\le i<j\le r),
\]
and thus $z_1,\ldots ,z_r$ are tightly constrained. Indeed, by applying the ideas of \S\S2 and 3 of this paper, we see 
that one should have $O(P^{\sqrt{4r+1}+\varepsilon})$ possible choices for $z_1,\ldots ,z_r$ constrained by these 
relations, rather than the upper bound of $O(P^r)$ employed in \S3. The reader may care to verify that such an 
enhancement of our basic strategy, if realisable, would replace our exponent $\sqrt{8k+9}-1$ in Theorems 
\ref{theorem1.1} and \ref{theorem1.2} by an improved exponent of size $O(k^{1/3})$.\par

The reality is surely far less generous than the scenario suggested in the previous paragraph. One could reasonably 
hope that the factors $z_l+z_m$ $(1\le l<m\le r)$ share common factors equitably amongst both themselves and the 
factors $u_{ij}$ with $i\ge 1$ and $1\le j\le r$. This line of reasoning suggests that once $z_1,\ldots ,z_r$ are fixed, 
then in the notation of \S3, one might have a bound roughly of the shape
\[
\biggl| \prod_{p=1}^rA_p\biggr| \ll P^{\kappa -\tfrac{1}{2}r^2}.
\]
One could then choose an index $p$ with $1\le p\le r$ in such a manner that
\[
|A_p|\ll P^{\tfrac{\kappa}{r}-\tfrac{r}{2}}.
\]
The exponent $\alpha_k$ in Theorem \ref{theorem1.1} would then be given by an expression roughly equal to
\[
\min_{2\le r\le 2k+1}\Bigl( \frac{r}{2}+\frac{2k+2}{r}\Bigr) \le \tfrac{1}{2}\sqrt{16k+17}.
\]
Thus, if realisable, this approach to improving the line of attack in \S3 would achieve a reduction in the exponent 
achieved in Theorem \ref{theorem1.1} of a factor $\sqrt{2}$ or thereabouts.\par

The latter improvement is speculative. An achievable improvement stems from a re-examination of the definition of 
$A_p$, with a more careful accounting of those factors $\alpha_{\mathbf i}$ having indices $\mathbf i$ with 
$i_m>i_p$ for $m\ne p$. Such a refinement is discussed in the setting of Vinogradov's mean value theorem in the 
final paragraph of \cite[\S2]{VW1997}. Recall that each variable $\alpha_{\mathbf i}$, with the property that $i_m=0$ 
for some index $m$ with $1\le m\le r$, is already fixed by our choice of $z_1,\ldots ,z_r$. So, in the setting of \S3, a 
treatment analogous to that concluding \cite[\S2]{VW1997} shows that we may work under the assumption that
\[
\biggl| \prod_{p=1}^rA_p\biggr| \le \biggl( \prod_{\substack{\mathbf i\in \mathcal I\\ i_m\ge 1\, (1\le m\le r)}}
|\alpha_{\mathbf i}|\biggr)^{r\psi(\kappa)/\kappa^r},
\]
where
\[
\psi(\kappa)=\sum_{i=1}^{\kappa-1}i^{r-1}<\kappa^r/r.
\]
Then, we can make a choice for $p$ with $1\le p\le r$ for which
\[
|A_p|\le \bigl( (2P)^\kappa \bigr)^{\psi(\kappa)/\kappa^r}.
\]
In the present setting, such a treatment would replace the conclusion of Theorem \ref{theorem1.1} with the bound
\[
V_k(P)-L_k(P)\ll P^{\beta_k+\varepsilon},
\]
where
\[
\beta_k=\min_{2\le r\le 2k+1}\biggl( r+\frac{1}{(2k+2-r)^{r-1}}\sum_{i=1}^{2k+1-r}i^{r-1}\biggr) .
\]
Thus, for example, one finds that $\beta_4=5-\tfrac{1}{7}$. For large values of $k$, however, the exponent $\beta_k$ 
is not visibly superior to the exponent $\alpha_k$ recorded in the statements of Theorems \ref{theorem1.1} and 
\ref{theorem1.2}.

\section{A generalisation of the Br\"udern-Robert system}
In this section we consider the generalisation of the system \eqref{1.3} given by the simultaneous equations
\begin{equation}\label{5.1}
\sum_{i=1}^{k+1} x_i^{(2j-1)d}=\sum_{i=1}^{k+1}y_i^{(2j-1)d}\quad (1\le j\le k),
\end{equation}
in which $d$ is a fixed natural number. We denote by $V_{k,d}^*(P)$ the number of integral solutions of the system 
\eqref{5.1} with $1\le x_i,y_i\le P$ $(1\le i\le k+1)$, and again denote by $L_k^*(P)$ the corresponding number of 
diagonal solutions defined as in the preamble to the statement of Theorem \ref{theorem1.2}. Of course, the system
\eqref{5.1} is obtained from \eqref{1.3} by replacing $x_i$ and $y_i$ by $x_i^d$ and $y_i^d$ throughout. It is 
therefore immediate from Theorem \ref{theorem1.2} that one has
\[
V_{k,d}^*(P)-L_k^*(P)\ll (P^d)^{\alpha_k+\varepsilon}\ll P^{d\sqrt{8k+9}-d+d\varepsilon}.
\]
Thus, the asymptotic formula $V_{k,d}^*(P)\sim (k+1)!P^{k+1}$ holds whenever $1\le d\le \sqrt{k/8}$. By a more 
careful analysis along the lines of \S\S2 and 3, a sharper conclusion may be obtained.

\begin{theorem}\label{theorem5.1}
When $k\ge 2$ and $d\ge 1$, one has
\[
V_{k,d}^*(P)-L_k^*(P)\ll P^{\alpha_{k,d}+\varepsilon},
\]
where
\[
\alpha_{k,d}=\min_{\substack{r\in \mathbb N\\ 2\le r\le 2k+1}}\Bigl( r-d+\frac{(2k+2)d}{r}\Bigr) .
\]
In particular, one has
\[
V_{k,d}^*(P)-L_k^*(P)\ll P^{\sqrt{8d(k+1)+1}-d+\varepsilon}.
\]
\end{theorem}

Let $\beta=3-2\sqrt{2}=0.1715\ldots $. Then it is not difficult to check that whenever one has $1\le d\le \beta(k-1)$, 
then
\[
\sqrt{8d(k+1)+1}-d<k+1,
\]
whence again $V_{k,d}^*(P)\sim (k+1)!P^{k+1}$. The conclusion of Theorem \ref{theorem5.1} therefore extends the 
range for $d$ in which the system \eqref{5.1} exhibits paucity from $1\le d\le \sqrt{k/8}$ to $1\le d\le \beta(k-1)$.

\begin{proof}[Proof of Theorem \ref{theorem5.1}]
We begin by rewriting the system of equations \eqref{5.1} in the form
\begin{equation}\label{5.2}
\sum_{i=1}^{2k+2}z_i^{2j-1}=0\quad (1\le j\le k),
\end{equation}
with $z_i=x_i^d$ and $z_{k+1+i}=-y_i^d$ $(1\le i\le k+1)$. In this apparition, we may again employ the identity 
supplied by Lemma \ref{lemma2.1} and its symmetrical brethren. This motivates the adoption once more of the 
notation \eqref{2.8}, \eqref{2.9} and \eqref{2.10}. Given a fixed choice of the $r$-tuple $\mathbf z=(z_1,\ldots ,z_r)$, 
we now denote by $\Psi_{k,d}(P;\mathbf z)$ the number of integral solutions of the simultaneous equations 
\eqref{2.11} and \eqref{2.12} with
\[
1\le |u_{ij}|\le 2P^d\quad (1\le i\le \kappa,\, 1\le j\le r).
\]
We then define
\[
\Psi_{k,d}(P)=\max_{\mathbf z}\Psi_{k,d}(P;\mathbf z),
\]
where the maximum is taken over all $r$-tuples $\mathbf z$ satisfying $z_i=\pm w_i^d$ for some integer $w_i$ 
with $1\le w_i\le P$ $(1\le i\le r)$ and
\[
z_l^2\ne z_m^2\quad (1\le l<m\le r).
\]
In the present scenario, we obtain an estimate analogous to that of Lemma \ref{lemma2.2}, now assuming the shape
\begin{equation}\label{5.3}
V_{k,d}^*(P)-L_k^*(P)\ll P^{r-1}+P^r\Psi_{k,d}(P).
\end{equation}
This we obtain by following the proof of Lemma \ref{lemma2.2}, mutatis mutandis, noting that the number of 
possible choices for $z_1,\ldots ,z_r$ to be counted is $O(P^r)$. We note that, implicit in these deliberations, we apply 
the symmetry of the equations \eqref{5.2} with respect to the variables $z_1,\ldots ,z_{2k+2}$. This symmetry may 
entail relabelling of variables in which $z_1,\ldots ,z_r$ are associated with choices of both the $x_i$ and the integers 
$-y_j$ for appropriate choices of $i$ and $j$.\par

It is evident that $\Psi_{k,d}(P;\mathbf z)\ll \Psi_k(P^d;\mathbf z)$. We are therefore led via the preamble to \eqref{3.2} 
to the estimate $\Psi_{k,d}(P;\mathbf z)\ll X_r(P^d)$, and thence via \eqref{5.3} and the argument leading to 
\eqref{3.7} to the bound
\[
V_{k,d}^*(P)-L_k^*(P)\ll P^{r-1}+P^rY_r(P^d).
\]
The discussion of \S3 leading to \eqref{3.9} consequently delivers the bound
\[
V_{k,d}^*(P)-L_k^*(P)\ll P^{r-1}+P^{r+\varepsilon}(P^d)^{\kappa/r}.
\]
On recalling from \eqref{2.10} once more that $\kappa=2k+2-r$, we conclude that
\[
V_{k,d}^*(P)-L_k^*(P)\ll P^{r-1}+P^{\alpha_{k,d}+\varepsilon},
\]
where $\alpha_{k,d}$ is the exponent introduced in the statement of Theorem \ref{theorem5.1}. This confirms the 
first assertion of Theorem \ref{theorem5.1}. The second assertion of Theorem \ref{theorem5.1} follows by applying 
Theorem \ref{theorem6.1} just as in the conclusion of the proof of Theorem \ref{theorem1.1}. This completes our 
proof of Theorem \ref{theorem5.1}.
\end{proof}

\section{Appendix: an elementary discrete inequality}
School students learn early in life that when $\lam>0$, the function $x+\lam/x$ achieves its 
minimum value for positive values of $x$ when $x=\sqrt{\lam}$, in which case the two 
terms comprising the function are equal. Indeed,
\[
x+\lam/x=2\sqrt{\lam}+\left( \sqrt{x}-\frac{\sqrt{\lam}}{\sqrt{x}}\right)^2\ge 
2\sqrt{\lam},
\]
and the conclusion is clear. If instead of minimising over all positive real values of $x$, one 
is restricted to work with positive integers, then one can approximate this argument by 
choosing $x$ to be one of the two positive integers closest to $\sqrt{\lam}$. A precise form 
of this conclusion is surely well-known to the cognoscenti, and was apparently known to 
this author 25 years ago (see \cite{VW1997}, and \cite{Woo2023} for a more recent 
application). The purpose of this appendix is to refresh the author's memory, while 
also making the conclusion more readily available.

\begin{theorem}\label{theorem6.1} When $\lam>0$, one has
\[
\min_{r\in \dbN}\, (r+\lam/r)\le \sqrt{4\lam+1},
\]
with equality if and only if $\lam=m(m-1)$ for some positive integer $m$.
\end{theorem}

\begin{proof} Let $r$ be the unique integer satisfying
\[
\sqrt{\lam+\tfrac{1}{4}}-\tfrac{1}{2}<r\le \sqrt{\lam+\tfrac{1}{4}}+\tfrac{1}{2}.
\]
Then we have $r+\lam/r\le \sqrt{4\lam+1}$ if and only if
\[
r^2+2\lam+\lam^2/r^2\le 4\lam+1,
\]
and this inequality holds if and only if $(r-\lam/r)^2\le 1$. Thus we see that the first 
conclusion of the theorem will be confirmed by verifying that $|r-\lam/r|\le 1$.\par

Put $\del=r-\sqrt{\lam+\tfrac{1}{4}}$, and note that 
$-\tfrac{1}{2}<\del \le \tfrac{1}{2}$. Then we have
\begin{align*}
\left| r-\frac{\lam}{r}\right| &
=\left| \frac{\left( \sqrt{\lam+\tfrac{1}{4}}+\del\right)^2-\lam}
{\sqrt{\lam+\tfrac{1}{4}}+\del}\right| \\
&=\left| 2\del +\frac{\tfrac{1}{4}-\del^2}{\sqrt{\lam+\tfrac{1}{4}}+\del}\right| .
\end{align*}
When $\del>0$, we now put $\tau=\tfrac{1}{2}-\del$. Then we find that
\[
0<2\del +\frac{\tfrac{1}{4}-\del^2}{\sqrt{\lam+\tfrac{1}{4}}+\del}\le 
1-2\tau+\frac{\tau-\tau^2}{1-\tau}\le 1.
\]
When $\del\le 0$, meanwhile, we put $\tau=\tfrac{1}{2}+\del$. A 
similar argument then yields
\[
-\left( 2\del +\frac{\tfrac{1}{4}-\del^2}{\sqrt{\lam+\tfrac{1}{4}}+\del}\right) 
\le 1-2\tau-\frac{\tau-\tau^2}{\sqrt{\lam+\tfrac{1}{4}}}\le 1
\]
and
\[
-\left( 2\del +\frac{\tfrac{1}{4}-\del^2}{\sqrt{\lam+\tfrac{1}{4}}+\del}\right) 
\ge 1-2\tau-\frac{\tau-\tau^2}{\tau}\ge -\tfrac{1}{2}.
\]
In either case, therefore, we have $|r-\lam/r|\le 1$. In view of our earlier discussion, this 
establishes the first conclusion of the theorem.\par

The upper bound asserted in the first conclusion of the theorem holds with equality if and 
only if there is an integer $r$ satisfying the equation $r+\lam/r=\sqrt{4\lam +1}$. This 
relation holds if and only if
\[
(r-\tfrac{1}{2}\sqrt{4\lam+1})^2=r^2-r\sqrt{4\lam+1}+\lam+\tfrac{1}{4}=
\tfrac{1}{4},
\]
and in turn, this equation holds if and only if
\[
r=\pm \tfrac{1}{2}+\tfrac{1}{2}\sqrt{4\lam+1}.
\]
Thus, one has $(2r\pm 1)^2=4\lam+1$, so that $\lam=r^2\pm r$. The first conclusion of 
the theorem consequently holds with equality if and only if $\lam=m(m-1)$ for some 
positive integer $m$.
\end{proof}

\bibliographystyle{amsbracket}
\providecommand{\bysame}{\leavevmode\hbox to3em{\hrulefill}\thinspace}

\end{document}